\documentclass[10pt,a4paper,reqno]{amsart} 
\usepackage{latexsym}
\usepackage{verbatim}
\usepackage{epsfig}
\usepackage{rotating}
\usepackage{amssymb}
\usepackage[T1]{fontenc}
\usepackage{afterpage}
\usepackage{color}
\usepackage{dsfont}
\usepackage{amssymb,amsfonts,amsmath,stmaryrd,bbm}
\usepackage[utf8]{inputenc}
\usepackage{pifont, tikz, subfigure}
\usetikzlibrary{plotmarks,shapes,arrows,positioning}
\usepackage{url}
\usepackage[plainpages=false,pdfpagelabels,colorlinks=true,citecolor=blue,hypertexnames=false]{hyperref}

\catcode`\@=11
\def\section{\@startsection{section}{1}%
 \z@{.7\linespacing\@plus\linespacing}{.5\linespacing}%
 {\normalfont\bfseries\scshape\centering}}

\def\subsection{\@startsection{subsection}{2}%
  \z@{.5\linespacing\@plus\linespacing}{.5\linespacing}%
  {\normalfont\bfseries\scshape}}

\def\subsubsection{\@startsection{subsubsection}{3}%
 \z@{.5\linespacing\@plus\linespacing}{-.5em}
 {\normalfont\bfseries}}
\catcode`\@=12

\newtheorem{theorem}{Theorem}
\newtheorem{lemma}[theorem]{Lemma}

\theoremstyle{definition}

\newcommand{\beq}{\begin{equation}}
\newcommand{\eeq}{\end{equation}}

\renewcommand{\epsilon}{\varepsilon}
\renewcommand{\iota}{\phi}

\makeatletter
\def\testb#1{\testb@i#1,,\@nil}%
\def\testb@i#1,#2,#3\@nil{%
  \draw[->, thick] (O) --++(#1);
  \ifx\relax#2\relax\else\testb@i#2,#3\@nil\fi}
\makeatother

\def\sequenceThreeD#1#2#3#4{%
  #2\quad\hbox to.6\hsize{$#3,\dots$\hfill}\quad\textrm{(#4)}
}

\long\def\greybox#1{%
    \newbox\contentbox%
    \newbox\bkgdbox%
    \setbox\contentbox\hbox to \hsize{%
        \vtop{
            \kern\columnsep
            \hbox to \hsize{%
                \kern\columnsep%
                \advance\hsize by -2\columnsep%
                \setlength{\textwidth}{\hsize}%
                \vbox{
                    \parindent=0bp
                    #1
                }%
                \kern\columnsep%
            }%
            \kern\columnsep%
        }%
    }%
    \setbox\bkgdbox\vbox{
        \pdfliteral{0.9 0.9 0.9 rg}
        \hrule width  \wd\contentbox %
               height \ht\contentbox %
               depth  \dp\contentbox
        \pdfliteral{0 0 0 rg}
    }%
    \wd\bkgdbox=0bp%
    \vbox{\hbox to \hsize{\box\bkgdbox\box\contentbox}}%
}

\newcommand{\ns}{\mathbb{N}}

\newcommand{\zs}{\mathbb{Z}}

\newcommand{\qs}{\mathbb{Q}}

\newcommand{\cs}{\mathbb{C}}
\newcommand{\fps}{formal power series}

\newcommand{\bx}{\bar x}
\newcommand{\by}{\bar y}

\newcommand{\cS}{\mathcal S}

\DeclareMathOperator{\Pol}{Pol}
\DeclareMathOperator{\Disc}{Disc}

\newcommand{\gf}{generating function}
\newcommand{\gfs}{generating functions}

\def\emm#1,{{\em #1}}

\newcommand{\T}{\tilde T}
\newcommand{\Tm}{T}

\newcommand{\Zm}{Z}

\newcommand{\Z}{\tilde Z}\newcommand{\U}{\tilde U}\newcommand{\V}{\tilde V}

\begin{document}
%
\title[Gessel's walks in the quadrant]
{An elementary solution of Gessel's walks\\ in the quadrant}

\author[M. Bousquet-M\'elou]{Mireille Bousquet-M\'elou}

\thanks{}
 
\address{CNRS, LaBRI, Universit\'e de Bordeaux, 351 cours de la
  Lib\'eration,  F-33405 Talence Cedex, France} 
\email{bousquet@labri.fr}

\begin{abstract}
 Around 2000, Ira Gessel conjectured that the number of lattice walks
 in the quadrant $\ns^2$,  starting and ending at the origin $(0,0)$ and  taking
 their steps in $\{
 \rightarrow,\nearrow,\leftarrow, \swarrow\}$  had a simple
hypergeometric form. In the following decade, this problem 
was recast in the systematic study of walks with small steps (that is,
steps in $\{-1,0,1\}^2$) confined to the quadrant.  The \gfs\ of such
walks are archetypal  solutions of \emm partial
 discrete differential equations,.  

A complete
classification of quadrant walks according to the nature of their \gf\
(algebraic, D-finite or not) is now available, but Gessel's
walks remained mysterious because they were the only model among the 23
D-finite ones that had not been given an elementary
solution. Instead, Gessel's conjecture was first proved using
an inventive  computer algebra approach in 2008.  A year later, the associated three-variate 
 \gf\ was proved to be algebraic by a computer algebra \emm tour de
 force,. This was re-proved  recently using  
 elaborate complex analysis machinery. We give here an
 elementary and constructive proof.
Our approach also solves other quadrant
 models (with multiple steps) recently proved to be algebraic via computer algebra.
\end{abstract}

\maketitle

\section{Introduction}
The enumeration of planar lattice walks confined to the quadrant has
received a lot of attention in the past decade.
The basic question 
reads as follows:  given a finite step set $\cS\subset \zs^2$ and a
starting point $P\in \ns^2$, what is the number $q(n)$ of $n$-step walks,
starting from $P$ and taking their steps in $\cS$, that remain in the
non-negative quadrant $\ns^2$? This is a generic and versatile
question, since such walks encode in a natural fashion many
discrete objects (systems of queues, Young tableaux and permutations
among others). More generally, the study of these walks fits in the
larger framework of walks confined to cones. They are also much
studied in probability theory, both in a discrete~\cite{duraj,FaIaMa99} and
a continuous~\cite{deblassie,garbit-raschel} setting. From a technical
point of view, counting walks in the quadrant is part of a general
program aiming at solving \emm partial discrete differential
equations,, see~\eqref{nseo} and~\eqref{eqQ} below. Details are given in Section~\ref{sec:pdde}, together with
a more algebraic viewpoint involving  division of power series.

On the combinatorics side, much attention has  focused on the \emm
nature, of the associated \gf\ $Q(t)=\sum_n q(n) t^n$. Is it 
rational in $t$, as  for unconstrained walks? Is it algebraic over
$\qs(t)$, as for walks confined to a (rational) half-space? More
generally, is it  D-finite, that is, a solution of a linear differential equation
with polynomial coefficients?  The answer depends on the
step set, and, to a lesser extent, on the starting point. 

A systematic
study was initiated in~\cite{BoMi10,mishna-jcta} for walks
starting at the origin $(0,0)$ and taking only \emm small, steps (that
is, $\cS\subset \{-1, 0, 1\}^2$). For these walks, a complete classification is now
available. In particular, the \gf\ $Q(t)$ (or rather, its three-variate
refinement $Q(x,y;t)$ that also records the coordinates of the endpoint of the walk)
is D-finite if and only if a certain group of rational transformations
is finite. The proof involves an attractive variety of tools, ranging
from basic  power series algebra~\cite{Bous05,BoMi10,mishna-jcta} to complex analysis~\cite{KuRa12,raschel-unified},
computer algebra~\cite{BoKa10,KaKoZe08} and number theory~\cite{BoRaSa12}.

In the  D-finite class lie, up to symmetries,
exactly 23 step sets $\cS$ (often called \emm models,) with small steps.
A uniform approach,  the \emm algebraic
kernel method,,  establishes the
D-finiteness of 19 of them~\cite{BoMi10}. These 19 models are
transcendental.
The remaining 
4 models are  structurally simpler, since they are algebraic,  but
they are also harder
to solve. Three of them, including the so-called Kreweras' model
$\cS=\{\leftarrow, \downarrow, \nearrow\}$, were solved in a uniform manner
in~\cite{BoMi10}, following several \emm ad hoc,
proofs~\cite{kreweras,gessel-proba,Bous05,bousquet-versailles}. The final one, $\cS=\{
\rightarrow,\nearrow, \leftarrow, \swarrow \}$, 
has been recognized as extremely challenging, and has been promoted
recently as a
popular model in the French magazine \emm Pour la science,~\cite{pls}.
It is
named after Ira Gessel, who conjectured, around 2000,
that for this model the
number of walks of length $2n$ ending at the origin is
\beq\label{Q00-conj}
16^n\, \frac{(5/6)_n (1/2)_n}{(5/3)_n (2)_n},
\eeq
where $(a)_n=a(a+1) \cdots (a+n-1) $ is the ascending
factorial. This is sequence A135404 is Sloane's \emm Encyclopedia of
integer sequences,~\cite{oeis}. 
 The associated \gf\ is algebraic, but no one noticed it at that time.
This attractive conjecture remained open for several years. In particular, it resisted the otherwise complete
solution of D-finite models presented in~\cite{BoMi10}. Finally, Kauers,
Koutschan and Zeilberger~\cite{KaKoZe08} succeeded in
proving~\eqref{Q00-conj} via a clever computer-aided approach. About a year later, this was generalized by
Bostan and Kauers~\cite{BoKa10}, who proved, by a \emm tour de force, in computer algebra,
that the \emm complete \gf, $Q(x,y;t)$ is also algebraic.
This series is defined as
\beq\label{complete}
Q(x,y;t)=\sum_{i,j,n\ge 0} q(i,j;n) x^i y^j t^n,
\eeq
where $q(i,j;n)$ is the number of $n$-step walks in the quarter plane
that start from $(0,0)$,  end at $(i,j)$ and take their steps in
$\cS$. A complete description, and rational parametrization of this
series is given in Theorem~\ref{thm:main} below. Both the result of~\cite{BoKa10} and
the immense computational effort it required were remarkable, but they
left open the problem of
finding a ``human'' proof of Gessel's conjecture, and of its
trivariate generalization. More recently, Bostan,
Kurkova and Raschel~\cite{BoKuRa13} proposed another proof, which  involves the
deep machinery of complex analysis first developed for stationary random walks~\cite{FaIaMa99},
and subsequently for counting walks in the quadrant~\cite{KuRa12,raschel-unified,FaRa10}. The solution
involves Weierstrass' $\wp$ and $\zeta$  functions, which
are transcendental, and  may appear a surprisingly complicated
detour to establish algebraicity. Let us also cite a number of other
attempts~\cite{ayyer-gessel,FaRa10,KaZe08,KuRa11,PeWi08,raschel-unified,sun}. Some
of them can be seen as preliminary steps to
the proofs of~\cite{KaKoZe08} or~\cite{BoKuRa13}. 

\medskip
In this paper, we present the first elementary solution of Gessel's
walks in the quadrant. It is elementary in the sense that it remains
at the level of \fps\ and polynomial equations. Also, it involves no
guessing nor high-level computer algebra\footnote{Although one feels
  better with a computer at hand when it comes  to handling a system
  of three polynomial equations. An accompanying Maple session is
  available on the author's \href{http://www.labri.fr/perso/bousquet/publis.html}{webpage}.}. Finally, it is constructive,
in the sense that we construct, incrementally, the extensions of
$\qs(x,y,t)$ that are needed to describe the series $Q(x,y;t)$. Even
if this sounds a bit technical at this stage, we can inform the
quadrant experts  that the key ingredient 
is that the
symmetric functions of the roots of the kernel are polynomials in
$1/x$. This is also the case for Kreweras' walks, and our proof is in
this sense close to one solution of Kreweras' model~\cite{bousquet-versailles}. 

\begin{theorem}\label{thm:main}
  The \gf\ $Q(x,y;t)$ is algebraic over $\qs(x,y,t)$, of degree $72$.
  The specialization $Q(0,0;t)$ has degree $8$, and can be written as
$$
Q(0,0;t)=\frac{32\,\Zm ^3(3+3\Zm -3\Zm ^2+\Zm ^3)}{(1+\Zm )(\Zm ^2+3)^3},
$$
where $Z=\sqrt T $ and $T$ is the only power series in $t$ with
constant term 
$1$
satisfying  
\beq\label{T-def}
T= 1+256\, t^2 \frac{T^3}{(T+3)^3}.
\eeq

The series $Q(xt,0;t)$ is an even series in $t$, with coefficients in
$\qs[x]$, and is cubic over $\qs(Z,x)$. It can be written as
$$
Q(xt,0;t)=\frac
{16\,T(U+UT-2T)M(U,Z)}{(1-T)(T+3)^3(U+Z)(U^2-9T+8TU+T^2-TU^2)}
$$
where
\begin{multline*}
  M(U,Z)=(T-1)^2U^3+Z(T-1)(T-16Z-1)U^2\\-TU(T^2+16ZT-82T-16Z+17)
-ZT(T^2-18T+128Z+81),
\end{multline*}
and $U$ is the only power series in $t$ with constant term $1$ (and
coefficients in $\qs[x]$) satisfying 
$$
16T^2(U^2-T)=
x(U+UT-2T)(U^2-9T+8TU+T^2-TU^2).
$$

Finally, the series $Q(0,y;t)$ is cubic over $\qs(Z,y)$. It can be
written as
$$
Q(0,y;t)=\frac{16\,VZ^3(3+V+T-VT)N(V,Z)}{(T-1)(T+3)^3(1+V)^2(1+Z+V-VZ)^2}
$$
where 
\begin{multline*}
N(V,Z)=(Z-1)^2(T+3)V^3+(T-1)(T+2Z-7)V^2\\
+(T-1)(T-2Z-7)V +(Z+1)^2(T+3)  ,
\end{multline*}
and $V$ is the only series in $t$, with constant term $0$,
satisfying 
$$
1-T+3V+VT= yV^2(3+V+T-VT).
$$
\end{theorem}
The algebraicity was established by Bostan and Kauers, and the
parametrization was given by van Hoeij in the appendix of their
paper~\cite{BoKa10} (with his notation, $T=v$, $U=u$ and $V=-w$). The
parametrization that  we construct through our proof is actually a bit different
(and of course equivalent), but we prefer to give van
Hoeij's to avoid confusion. Just for the record, ours reads as
follows:
\beq\label{T-param-original}
\T= t^2(1-\T)(1+3\T)^3, \qquad \qquad \Z=\T(1-\Z+\Z^2),
\eeq
$$
\U =1+x\Z \frac{(\U +\Z -\U \Z +\U \Z ^2)(\U -\U \Z -\Z ^3+\U \Z ^2)}{\U (\Z -1)(\Z +1)^3(\Z ^2+\U )},
$$
$$
\V =y\frac{(\V -1)^2(1-\Z -\V +\Z ^2-\Z ^3)}{(1-\Z +\V -\Z ^2)^2},
$$
with constant terms $0$ for $\T$ and  $\Z$, 1 for $\U $ and $y/(1+y)$
for $\V$.

\medskip
We conclude this introduction with some notation.
For a ring $R$, we denote by $R[x]$
the ring of
polynomials 
in $x$ with coefficients in
$R$. If $R$ is a field, then $R(x)$ stands for the field of rational functions
in $x$. 
Finally, if $F(x;t)$  is a power series in~$t$ whose coefficients are
Laurent polynomials in $x$, say
$$
F(x;t)=\sum_{n\ge0, \ i\in \zs}f(i;n) t^n x^i,
$$
 we denote by $[x^{\ge}]F(x;t)$ the \emm non-negative part of $F$ in $x$,:
$$
 [x^{\ge}]F(x;t):= \sum_{n\ge 0,\  i\in \ns}  f(i;n) t^n x^i.
$$
This generalizes the standard notation $[x^i]F(x;t)$ for the
coefficient of $x^i$ in $F$. We  similarly define the non-positive part
of $F(x;t)$, denoted by $[x^\le]F(x;t)$.

\section{Context: partial discrete differential equations}
\label{sec:pdde}
The starting point of the systematic approach to the enumeration of
quadrant walks is a functional equation that characterizes their complete
\gf \ $Q(x,y;t)\equiv Q(x,y)$, defined as in~\eqref{complete}. For
instance, for square lattice walks (with steps $\rightarrow, \uparrow,
 \leftarrow,\downarrow $) starting at $(0,0)$, this equation
reads:
\beq\label{nseo}
Q(x,y)=1+ t(x+y) Q(x,y) + t\, \frac{Q(x,y)-Q(0,y)}x +
t\,\frac {Q(x,y)-Q(x,0)}y.
\eeq
It translates the fact that, to construct a  walk of length
$n\ge 1$, we simply add a step to a shorter walk of length $n-1$. The
divided differences (or \emm discrete derivatives,)
\beq\label{dd}
\frac{Q(x,y)-Q(0,y)}x \qquad\hbox{and} \qquad
\frac {Q(x,y)-Q(x,0)}y
\eeq
arise from the fact that  one cannot
 add a West step if the shorter walk  ends on the $y$-axis,
nor a South step  if it ends on the $x$-axis. For Gessel's walks, with
steps $ \rightarrow,\nearrow,\leftarrow, \swarrow$, the corresponding
equation reads
\begin{multline}\label{eqQ}
  Q(x,y)=1+ t(x+xy) Q(x,y) + t\, \frac{Q(x,y)-Q(0,y)}x\\ + t\,
\frac{Q(x,y)-Q(x,0)-Q(0,y)+Q(0,0)}{xy},
\end{multline}
because the $\swarrow$ step cannot be appended to a walk ending on the
$x$- or $y$-axis. Refering to terms like~\eqref{dd} as \emm discrete
derivatives,, it makes sense to call such equations \emm partial discrete
differential equations,. 

The more algebraically inclined reader will write~\eqref{nseo} as
$$
xy=\left(xy-t(x^2y+xy^2+x+y)\right)Q(x,y)+txQ(x,0)+tyQ(0,y),
$$
and observe that we are just performing the formal division of
$xy$ by the polynomial $\left(xy-t(x^2y+xy^2+x+y)\right)$ with initial
monomial $xy$ (this is an instance of Grauert-Hironaka-Galligo
division, see e.g.~\cite[Thm.~10.1]{rond}).   

Discrete differential  equations are ubiquitous in combinatorial enumeration. For
instance, for the much simpler problem of counting walks with  steps
$\pm1$  starting at $0$ and remaining at a non-negative level, the
corresponding functional equation reads
$$
F(x)\equiv F(x;t)= 1+ tx F(x) + t\, \frac{F(x)-F(0)}x,
$$
an \emm ordinary, discrete differential equation (since the
derivatives are only taken with respect to $x$). More generally, the
enumeration of lattice walks confined to a rational half-space
systematically yields \emm linear ordinary, discrete differential
equations (DDEs)~\cite{BaFl02,bousquet-petkovsek-1}. Note that the
order of the equation increases with the size of the down steps: if,
in the above problem, we replace the step $-1$ by a step $-2$, the
equation reads
$$
F(x)=1+txF(x)+ t\, \frac{F(x)-F(0)-xF'(0)}{x^2},
$$
and involves a discrete derivative of order 2.  Moving beyond linear
equations, the enumeration of
\emm maps, (connected graphs properly embedded in a surface of prescribed
genus) has produced since the sixties a rich collection of \emm
non-linear, ordinary discrete differential
equations (see e.g.~\cite{bender-canfield,brown-tutte-non-separable,tutte-triangulations}). For instance, the enumeration of planar maps 
is
governed by the following equation:
$$
F(x;t)= F(x)= 1+ tx^2F(x)^2+ tx\, \frac{xF(x)-F(1)}{x-1}.
$$
(Note that the discrete derivative is now taken at $x=1$.) 

An approach for
solving linear  ordinary DDEs was initiated by Knuth
in the seventies~\cite[Section~2.2.1, Ex.~4]{knuth} and is now known as the \emm kernel
method,~\cite{hexacephale,bousquet-petkovsek-1}. Another approach, developed by Brown in the sixties and
known as the \emm quadratic method,, solves  quadratic ordinary 
DDEs of first order~\cite{brown-square,goulden-jackson}. 
More recently, this approach was generalized by the author
and A.~Jehanne~\cite{mbm-jehanne} to ordinary DDEs of any degree and order, with the
striking result that \emm any series solution of a (well-founded) equation
of this type is algebraic,\footnote{As learnt recently by the
  author, this algebraicity also follows from a difficult theorem on
  Artin's approximation with \emm nested
  conditions,, see~\cite{kurke} or~\cite[Thm.~1.4]{popescu}.}. This result is of particular
importance to this paper, since the core of our proof is to derive
from the \emm partial, DDE~\eqref{eqQ} satisfied by $Q(x,y)$ an \emm ordinary,
DDE satisfied by $Q(x,0)$. This is achieved
in Section~\ref{sec:ordinaryDDE}. This ordinary DDE  (of order 3
and degree 3)  is then 
solved using the method of~\cite{mbm-jehanne}, and found, of course,
to have an algebraic solution.    

\medskip
If ordinary DDEs are now well-understood, much less is known about
partial DDEs, even linear ones. The enumeration of quadrant walks with
small steps
appears as an essential step in their study, because the resulting
equations are  as simple as possible:  linear, first order, and
 involving derivatives with respect to two variables only. This study
 started  around 2000, and it was soon understood that the
 solutions would not always be algebraic, nor even D-finite~\cite{BoPe03,MiRe09}.
 The classification of these problems is now complete. That is, one knows
 for every step set $\cS\subset \{-1,0,1\}^2$ if the series $Q(x,y;t)$
 is rational, algebraic, D-finite or not~\cite{BoKa10,BoRaSa12,BoMi10,KuRa12,MiRe09,MeMi13}. A crucial role is played by
 a certain group of rational transformations associated with $\cS$:
 the series $Q(x,y;t)$ turns out to be D-finite if and only if this
 group is finite. In this case, a generic method solves the partial
 DDE associated with the problem~\cite{BoMi10}, except for four sets $\cS$ for which
 the series $Q(x,y;t)$ is in fact algebraic. Three of these four problems were solved
 in a uniform way in~\cite{BoMi10}. The fourth one is Gessel's
 model, the story of which we have told in the introduction. This
 paper thus adds a building block to the classification of 
 quadrant walks, by deriving from the partial DDE~\eqref{eqQ}, in a
 constructive manner,  an
 algebraic equation for $Q(x,y)$.

To finish, let us mention a few investigations on more general partial
DDEs: quadrant walks with larger backwards steps yield linear
equations of larger
order~\cite{FaRa15}; coloured planar maps yield non-linear
 equations~\cite{bernardi-mbm-alg,mbm-bcc}; walks confined
to the non-negative octant $\ns^3$ yield linear  equations with
three differentiation variables~\cite{BoBoKaMe};  the famous enumeration of permutations
with bounded ascending subsquences (and in fact, many problems related
to Young tableaux, plane partitions or alternating sign matrices~\cite{zeilberger-asm})
yields linear equations with an arbitrary number of differentiation
variables~\cite{mbm-monotone}.

\section{The proof}
We begin, in a standard way, by writing a functional equation
satisfied by $Q(x,y;t)$. It is based on a step by step construction
of walks~\cite{BoMi10}: 
$$
Q(x,y)=1+t(\bx+\bx\by+x+xy)Q(x,y)
-t\bx(1+\by)Q(0,y)-t\bx\by\left(Q(x,0)-Q(0,0)\right),
$$
where $Q(x,y)\equiv Q(x,y;t)$, $\bx=1/x$ and $\by=1/y$. Equivalently,
\begin{equation}
 xyK(x,y) Q(x,y) = xy-t(Q(x,0)-Q(0,0))-t(1+y)Q(0,y) ,
\label{qdp-gessel1}
\end{equation}
where 
$$K(x,y)=1-t(\bx+\bx\by+x+xy)
$$
is the \emm kernel, of the equation. 
%
This Laurent polynomial is left invariant by the following two rational
transformations:
$$
  \Phi: (x,y) \mapsto (\bx\by,y) \quad \hbox{and} \quad 
\Psi:=(x,y) \mapsto (x,\bx^2\by).
$$
Both are involutions, and they generate a group $G$ of order 8:
\begin{multline*}
(x,y)  
 {\overset{\Phi}{\longleftrightarrow}} (\bx \by ,y)
 {\overset{\Psi}{\longleftrightarrow}} (\bx\by,x ^2y)
 {\overset{\Phi}{\longleftrightarrow}} (\bx,x^2y)
 {\overset{\Psi}{\longleftrightarrow}} (\bx,\by)
 {\overset{\Phi}{\longleftrightarrow}}\\ (xy ,\by)
 {\overset{\Psi}{\longleftrightarrow}} (xy,\bx ^2\by)
 {\overset{\Phi}{\longleftrightarrow}} (x,\bx^2\by)
 {\overset{\Psi}{\longleftrightarrow}} (x,y).
\end{multline*}
The construction of this group is also standard (see for instance~\cite{BoMi10}).

\subsection{Canceling the kernel}
As a polynomial in $y$, the kernel $K(x,y)$ has two roots, which are
series in $t$ with coefficients in $\qs[x,\bx]$:
\begin{align*}
Y_0(x)&=\displaystyle \frac{1-t(x+\bx) -\sqrt
{ (1-t(x+\bx))^2-4t^2}
}{2tx}=
 &\bx t
+O(t^2) , \\
\\
 Y_1(x)&=\displaystyle \frac{1-t(x+\bx) +\sqrt
{ (1-t(x+\bx))^2-4t^2}
}{2tx}
=\displaystyle\frac \bx t -(1+\bx ^2)\hskip -4mm &-\ \bx t + O(t^2).
\end{align*}
Observe  that the series  $xY_i(x)$ are symmetric in $x$ and $\bx$:
$$\bx Y_i(\bx)=xY_i(x).$$
Moreover, the elementary symmetric functions of the $Y_i$, namely
\begin{equation}
Y_0+Y_1= -1+\frac \bx t -\bx^2 \qquad \hbox{and}
\qquad Y_0Y_1 = \bx^2 ,
\label{symmetricG}
\end{equation}
 are polynomials in $\bx =1/x$. This property also holds for
Kreweras walks~\cite{bousquet-versailles,Bous05}, and plays a crucial role in our proof. The following
lemma tells us how to extract the constant term of a symmetric
polynomial in $Y_0$ and $Y_1$.

\begin{lemma}
\label{lemma-symmetric-G}
Let $P(u,v) $ be a symmetric
polynomial in $u$ and $v$ with coefficients in~$\qs$. Then
$P(Y_0,Y_1)$ is a 
polynomial  in $\bx$ with coefficients in $\qs[1/t]$. Its constant
term in $x$ (equivalently, its
non-negative part in $x$)
is $P(0,-1)$.
\end{lemma}
\begin{proof} The first statement follows from~\eqref{symmetricG} and the fact that every symmetric
   polynomial in $u$ and $v$ is a polynomial in $u+v$ and $uv$. 
  By linearity, it suffices to check the second statement when
  $P(u,v)=u^mv^n+u^nv^m$. If $\min(m,n)>0$, then $P(Y_0,Y_1)$ has a
  factor $Y_0Y_1=\bx^2$ and its constant term is $0$. If
  $P(u,v)=u^n+v^n$, then one proves by induction on $n$ that the
  constant term is~$2$ if $n=0$, and $(-1)^n$ otherwise.
\end{proof}
\noindent{\bf Application.} We will typically apply this lemma to prove that if  $F(u,v;t)$ is
a series in $t$ whose coefficients are symmetric polynomials in $u$
and $v$, with rational coefficients, and if $F(Y_0,Y_1;t)$ is
well-defined as a series in $t$ (with coefficients in $\qs[x,\bx]$),
then its coefficients actually lie in $\qs[\bx]$.

\medskip
 Consider now the orbit of $(x,Y_0)$ under the action of the group
 $G$:
\begin{multline*}
(x,Y_0)  
 {\overset{\Phi}{\longleftrightarrow}} (xY_1 ,Y_0)
 {\overset{\Psi}{\longleftrightarrow}} (xY_1,x ^2Y_0)
 {\overset{\Phi}{\longleftrightarrow}} (\bx,x^2Y_0)
 {\overset{\Psi}{\longleftrightarrow}} (\bx,x^2Y_1)
 {\overset{\Phi}{\longleftrightarrow}}\\ (xY_0 ,x^2Y_1)
 {\overset{\Psi}{\longleftrightarrow}} (xY_0,Y_1)
 {\overset{\Phi}{\longleftrightarrow}} (x,Y_1).
\end{multline*}
%
By construction, each pair
 $(x',y')$ in this orbit cancels the kernel $K(x,y)$. If moreover the
 series $Q(x',y')$ is well-defined, then we derive from~\eqref{qdp-gessel1} that
$$
R(x')+S(y')=x'y',
$$
where we use the notation
\beq\label{RS-def}
R(x)=t(Q(x,0)-Q(0,0)), \qquad S(y)=t(1+y)Q(0,y) .
\eeq

Since $Y_0$ is a power series in $t$, the pairs $(x,Y_0)$ and $(\bx,
x^2Y_0)$ can obviously be substituted for $(x,y)$ in $Q(x,y)$. Given
that $xY_0$ is symmetric in $x$ and $\bx$, these
pairs are derived from one another by replacing $x$ by $\bx$. One has
to be more careful with pairs involving $Y_1$, since this series
contains a term $\bx/t$. With the step set that we consider, and the
non-negativity conditions imposed by the quadrant, one easily
checks that each monomial $x^i y^j t^n$ occurring in the series
$Q(x,y)$ satisfies $n+i-j \ge n/2$. Given that $Y_0=\Theta(t)$ and
$Y_1=\Theta(1/t)$, this implies that the pair $(xY_0,Y_1)$ and its
companion $(xY_0,x^2Y_1)$ (obtained by replacing $x$ by $\bx$) can be
substituted for $(x,y)$ in $Q(x,y)$, so that $Q(xY_0,Y_1)$ is a series
in $t$ with coefficients in $\qs[x, \bx]$. We will not use the other
pairs of the orbit, and the reader can check that they do not give
well-defined series $Q(x',y')$.

We thus obtain a total of four equations:
\begin{align}
  R(x)+S(Y_0) & = xY_0 \label{eqG1}\\
R(xY_0)+S(Y_1) & = \bx,\\
R(\bx)+S(x^2Y_0) & = xY_0,\\
R(xY_0)+S(x^2Y_1) & = x.
\end{align}


\subsection{An equation relating $\boldsymbol{R(x)}$ and
  $\boldsymbol{R(\bx)}$}
We will now construct from the above system
two
identities that are symmetric in $Y_0$ and $Y_1$, and extract their
non-negative parts in $x$.

We first  sum  the first two equations, and subtract the  last two:
$$
R(x)-S(x^2Y_0)-S(x^2Y_1)+x = R(\bx)-S(Y_0)-S(Y_1)+\bx .
$$
For two indeterminates $u$ and $v$, the coefficient of $t^n$ in  the
series $S(u)+S(v)$ is a symmetric polynomial in $u$ and $v$. Applying
Lemma~\ref{lemma-symmetric-G} to this coefficient (for any $n$) shows that
the right-hand side of the above identity is a series in $t$ with coefficients in
$\qs[\bx]$. But the left-hand side is obtained by replacing $x$ by $\bx$ in
the right-hand side. This implies that both sides are  independent of
$x$ and  equal to their constant term, that is, to
$-S(0)-S(-1)$ (by~\eqref{RS-def} and  Lemma~\ref{lemma-symmetric-G} again). Finally, since $S(y)$ is a
multiple of $(1+y)$ (see~\eqref{RS-def}), this constant term is simply $-S(0)$. 

We have thus obtained a new equation,
$$
S(Y_0)+S(Y_1)=R( \bx)+\bx +S(0).
$$
Combined with~\eqref{eqG1} and~\eqref{symmetricG}, 
it gives
\beq\label{eq5G}
S(Y_1)-xY_1=R(x)+R(\bx)+2\bx-1/t+x+S(0).
\eeq

For our second symmetric function of $Y_0$ and $Y_1$,
we take a product derived from~\eqref{eqG1} and~\eqref{eq5G}:
\beq\label{eq-num}
(S(Y_0)-xY_0)(S(Y_1)-xY_1)= -R(x)\left( R(x)+R(\bx)+2\bx-1/t+x+S(0)\right).
\eeq
We want to extract the non-negative part in $x$. Let us
focus first on the left-hand side. By Lemma~\ref{lemma-symmetric-G},
the term $S(Y_0)S(Y_1)$  contributes $S(0)S(-1)$, which is zero since
$S(y)$ is a multiple of $(1+y)$. Then $x^2Y_0Y_1$ equals 1
by~\eqref{symmetricG}. We are left with the term
$-x(Y_0S(Y_1)+Y_1S(Y_0))$. By Lemma~\ref{lemma-symmetric-G}, the factor
between parentheses is a series in $t$ with polynomial coefficients in
$\bx$, and its constant term in $x$ is $-S(0)$. But we also need to
determine the coefficient of $\bx$ in this series. Expanding
$S(y)$ in powers of $y$ shows that
$$
Y_0S(Y_1)+Y_1S(Y_0)= (Y_0+Y_1)S(0) +Y_0Y_1 F(Y_0,Y_1),
$$
for some series $F(u,v)$ in $t$ with symmetric coefficients in $u$ and
$v$. Since $Y_0Y_1=\bx^2$, the coefficient of $\bx$ in
$Y_0S(Y_1)+Y_1S(Y_0)$ is the same as in $(Y_0+Y_1)S(0)$, namely
$S(0)/t$. We can now extract the non-negative part from~\eqref{eq-num}, and
this allows us to express the
non-negative part of $R(x)R(\bx)$:
$$
1+(x-1/t)S(0)= -R(x)^2-[x^{\ge}] R(x)R(\bx)-(2\bx-1/t+x+S(0)) R(x).
$$
(Recall that $R(x)$ is a multiple of $x$.)  Extracting the constant term in $x$
gives
$$
1-S(0)/t=-[x^0]R(x)R(\bx)-2R'(0).
$$
Since $R(x)R(\bx)$ is  symmetric in $x$ and $\bx$, we can now
reconstruct it:
\begin{align*}
  R(x) R(\bx)&= [x^{\ge}] R(x)R(\bx)+[x^{\le}]
  R(x)R(\bx)-[x^0]R(x)R(\bx)\\
&=
 -R(x)^2-(2\bx-1/t+x+S(0)) R(x)-R(\bx)^2\\
&\hskip 4mm -(2x-1/t+\bx+S(0)) R(\bx)
-1-(\bx+x-1/t)S(0)
+2R'(0).
\end{align*}
That is,
\begin{multline}\label{eqRquad}
  R(x)^2+R(x)R(\bx)+R(\bx)^2
+(2\bx-1/t+x+S(0)) R(x)\\+(2x-1/t+\bx+S(0))
R(\bx)=2R'(0)-(\bx+x-1/t)S(0)-1.
\end{multline}

\subsection{An equation for $\boldsymbol{R(x)}$ only}
 \label{sec:ordinaryDDE}
We cannot extract the positive part of the above equation explicitly
because of  the ``hybrid'' term $R(x)R(\bx)$. However, the form
of the first three terms  suggests a multiplication by
$R(x)-R(\bx)$ to decouple the series in $x$ from the 
series in $\bx$. More
precisely, if we multiply~\eqref{eqRquad} by $R(x)-R(\bx) +\bx -x$, we
find a decoupled equation
\beq\label{Peq}
P(x)=P(\bx),
\eeq
with
\begin{multline}\label{P-def}
  P(x)= R(x)^3+ (S(0)+3\bx -1/t) R(x)^2\\
+\left( 2\bx^2-\bx/t+x/t-x^2-2R'(0)+(2\bx
  -1/t)S(0)\right)R(x)\\-x^2S(0)+x\left( 2R'(0)+S(0)/t-1\right).
\end{multline}
Given that $R(x)$ is a multiple of $x$, all terms in the expansion of
$P(x)$ have an exponent (of $x$) at least $-1$. But then~\eqref{Peq}
implies that $P(x)$ is a symmetric Laurent polynomial in $x$, of degree 1 and
valuation $-1$. More precisely,
$$
P(x)=[x^0]P(x) +(x+ \bx) [x]P(x)=[x^0]P(x) +(x+ \bx) [\bx]P(x).
$$
Thus, by expanding the expression \eqref{P-def} of $P(x)$ in $x$ at
order $0$, we 
find:
$$
P(x)= 2(x+\bx) R'(0) + R'(0)(2S(0)-1/t) +R''(0) .
$$
Returning to~\eqref{P-def}, this gives
\begin{multline}\label{R-cat}
  R(x)^3+ (S(0)+3\bx -1/t) R(x)^2\\
+\left( 2\bx^2-\bx/t+x/t-x^2-2R'(0)+(2\bx
  -1/t)S(0)\right)R(x)\\
=R''(0)+R'(0)(2S(0)+2\bx -1/t) +xS(0)(x-1/t)+x.
\end{multline}
Thus $R(x)$ satisfies a cubic equation over
$\qs(t,x,S(0),R'(0),R''(0))$, 
$$
\Pol(R(x), S(0),R'(0),R''(0),t,x)=0,
$$
with 
\begin{multline}\label{Pol-def}
  \Pol(x_0,x_1, x_2, x_3, t, x)=
  x_0 ^3+ (x_1+3\bx -1/t) x_0 ^2\\
+\left( 2\bx^2-\bx/t+x/t-x^2-2x_2 +(2\bx
  -1/t)x_1\right)x_0 \\
-x_3-x_2 (2x_1+2\bx -1/t) -xx_1(x-1/t)-x.
\end{multline}
In particular, if we prove that the one-variable series $S(0)$,
$R'(0)$ and $R''(0)$ are algebraic over $\qs(t)$, then~\eqref{R-cat}
shows that $R(x)$ is algebraic over $\qs(t,x)$. Observe, as an
encouraging sign, that $R(x)=t(Q(x,0)-Q(0,0))$ has degree (at most) $3$ over
$\qs(t,x,S(0), R'(0), R''(0))$, in accordance with
Theorem~\ref{thm:main}.
Equation~\eqref{R-cat} is, in disguise, an ordinary discrete
differential equation of degree 3 and order 3, as discussed in
Section~\ref{sec:pdde}.

\subsection{The generalized quadratic method}
\label{sec:quad}
We have described in~\cite{mbm-jehanne} how to study  equations of the form
\beq\label{quad1}
\Pol(R(x), A_1, \ldots, A_k,t,x)=0,
\eeq
where $\Pol(x_0, x_1, \ldots, x_k, t,x)$ is a polynomial with
(say) rational coefficients, $R(x)\equiv R(x;t)$ is a \fps\ in $t$ with coefficients in
$\qs[x]$, and $A_1, \ldots, A_k$ are $k$ auxiliary series depending
on  $t$ only (in the above example, $\Pol$ is a \emm Laurent, polynomial in
$t$ and $x$, but this makes no difference). 
The strategy of~\cite{mbm-jehanne} instructs us to look for power series
$X(t)\equiv X$ satisfying
\begin{equation}
\frac{\partial \Pol }{\partial x_0}(R(X), A_1, \ldots, A_k,t,X)=0.
\label{quad2} \end{equation}
Indeed, by differentiating~\eqref{quad1} with respect to $x$, we see
that any such series also satisfies
\begin{equation}  
 \frac{\partial \Pol }{\partial x}(R(X), A_1, \ldots, A_k,t,X)=0,
\label{quad3} \end{equation}          
and we thus obtain  three polynomial equations, namely
Eq.~\eqref{quad1} 
written for $x=X$, Eqs.~\eqref{quad2} and~\eqref{quad3}, that relate
the $(k+2)$ unknown series $R(X)$, $A_1, \ldots, A_k$ and~$X$.
 If we
can prove the existence of $k$ distinct series $X_1, \ldots, X_k$
satisfying~\eqref{quad2}, we will 
have $3k$ equations between the $3k$ unknown series $R(X_1), \ldots ,
R(X_k)$, $A_1, \ldots, A_k$, $X_1, \ldots, X_k$. If there is no redundancy
in this system, we will have proved 
that each of the $3k$ unknown series is algebraic over
$\qs(t)$. 

\medskip
We apply this strategy to~\eqref{R-cat}, with $A_1=S(0)$, $A_2=R'(0)$ and
$A_3=R''(0)$. Equation~\eqref{quad2} reads:
\begin{multline}\label{quad2G}
  3  R(X)^2+2 (S(0)+3/X -1/t) R(X)\\
+2/X^2-1/(tX)+X/t-X^2-2R'(0)+(2/X  -1/t)S(0)=0.
\end{multline}
Recall that $R(x):=t(Q(x,0)-Q(0,0))$ and $S(0):=tQ(0,0)$ are multiples
of $t$. Hence, once multiplied by $tX^2$,  this equation has the following form: 
$$
X(1-X)(1+X)=t \Pol_1(Q(X,0), R'(0), Q(0,0), t,X).
$$
 This shows that there exists exactly three series in $t$, denoted
$X_0$, $X_1$ and $X_2$, that cancel~\eqref{quad2G}. Their constant
 terms are respectively 0, 1 and $-1$. Due to the special
 form of our polynomial $\Pol$ (given by~\eqref{Pol-def}), it is in fact simple to determine these
series. Indeed, we observe that
\beq\label{simple}
tx^2 \left( \frac{\partial \Pol}{\partial x_0} + x^2  \frac{\partial
  \Pol}{\partial x} \right)= (1-x)(1+x)(2tx^2+2t-x)(x_0 x+x_1x+1),
\eeq
and since the three series $X_i$ cancel both partial derivatives of $\Pol$,
we conclude that 
$
X_1=1$, $X_2=-1$, while $X_0$ is the only formal power series in $t$ that cancels
$(2tX^2+2t-X)$ (a Catalan-ish series, the exact expression of which we
will not use). The fourth factor in~\eqref{simple} cannot vanish for
$x_0=R(X)$ and $x_1=S(0)$ because of the factor $t$ occurring in these series.

Let $\Disc(x)$ be the discriminant of $\Pol(x_0, S(0), R'(0),
R''(0),t,x)$ with respect to $x_0$. It is obtained by eliminating $x_0$
between this polynomial and its derivative with respect to $x_0$.  Then $\Disc(X_i)=0$ for
$i=0,1,2$, and  this gives  three polynomial
equations relating  $S(0)$, $R'(0)$, $R''(0)$, $t$ and $X_0$. In
fact, we observe that $\Disc(x)$
   is
symmetric in $x$ and $\bx$ (this comes from the construction
of~\eqref{R-cat}) and is thus a polynomial in $s=x+\bx$. So  we can alternatively  write that this polynomial  vanishes at $s=\pm 2$ and $s=1/(2t)$.
Systematically   eliminating two of the three  series $S(0)$,
$R'(0)$ and $R''(0)$ in the resulting system of
three equations  yields an algebraic equation for
each of them:   for $R'(0)$ this
equation has degree 4, 
and degree 8 for the other two series. The equations for $R'(0)$ and $S(0)$ are
shown below. From this, one can derive an
algebraic equation for $R(x)$ (using~\eqref{R-cat}), that is, for $Q(x,0)$,
and then one for $S(y)=t(1+y)Q(0,y)$,
as explained below. Finally, the original functional equation~\eqref{qdp-gessel1}
proves the algebraicity of $Q(x,y)$.

\subsection{Rational parametrization}
 To avoid handling big polynomials, it is convenient to use
rational parametrizations of all these series. To begin with, the equation satisfied by
$R'(0)$ is:
\begin{multline*}
  729t^6R'(0)^4+243t^4(4t^2+1)R'(0)^3-27t^2(14t^4+19t^2-1)R'(0)^2\\
-(20t^2-1)(7t^2-6t+1)(7t^2+6t+1)R'(0)-t^2(343t^4-37t^2+1)=0.
\end{multline*}
It  has genus 0 ---so Maple
tells us--- and can be parametrized by introducing
the unique \fps \  $T  \equiv T (t)$  in
$t$ with constant term $1$ 
satisfying~\eqref{T-def}.
Then 
$$
R'(0)=\frac{(\Tm-1)(21-6\Tm+\Tm^2)}{(\Tm+3)^3}.
$$ 
Rational  parametrizations can be  computed via
the Maple command {\tt parametrization}. 
The one that we originally obtained has a slightly different
form, see~\eqref{T-param-original}. 

The equation satisfied by $S(0)$ is:
\begin{multline*}
  27t^7S(0)^8+108t^6S(0)^7+189t^5S(0)^6+189t^4S(0)^5-9t^3(32t^4+28t^2-13)S(0)^4
\\
-9t^2(64t^4+56t^2-5)S(0)^3-2t(256t^6-312t^4+156t^2-5)S(0)^2\\
-(32t^2-1)(4t^2-6t+1)(4t^2+6t+1)S(0)-t(256t^6+576t^4-48t^2+1) =0.
\end{multline*}
Since $S(0)=tQ(0,0)$, this gives an equation for $Q(0,0)$, involving
only even powers of $t$. If we replace $t^2$ by its rational
expression in $T$ derived from~\eqref{T-def}, this equation of degree
8 in $Q(0,0)$ factors into
a quadratic term and one of degree 6. Injecting the first
few coefficients of $Q(0,0)$ shows that the term that vanishes is the
quadratic one. So $Q(0,0)$ has degree $2$ over $\qs(T)$, and can be
written in terms of $Z=\sqrt T$ as stated in Theorem~\ref{thm:main}:
$$
Q(0,0)=\frac{32\,\Zm ^3(3+3\Zm -3\Zm ^2+\Zm ^3)}{(1+\Zm )(\Zm ^2+3)^3}.
$$
%
%
The series $R''(0)/t$ is also found to  have a  rational expression in
$Z$: 
$$
R''(0)=2t [x^2]Q(x,0)=1024\,t\,\frac{ \Zm ^3(\Zm -1) (1+2\Zm +7\Zm ^2-\Zm ^4-2\Zm ^5+\Zm ^6)}{(1+\Zm )(3+\Zm ^2)^6}.
$$

We now go back to the equation~\eqref{R-cat} satisfied by
$R(x)=t(Q(x,0)-Q(0,0))$. Injecting  the above expressions of
$S(0)=tQ(0,0)$, $R'(0)$ and $R''(0)$  in it gives a cubic equation for
$Q(x,0)$ over $\qs(t,Z,x)$. We now consider $Q(xt,0)$ instead of
$Q(x,0)$, since it is an even series in $t$. Using~\eqref{T-def} to
express $t^2$ in terms of $T=Z^2$, we
see that this series is cubic over $\qs(Z,x)$. 

We finally construct an equation for $ Q(0,y)$ 
thanks to the kernel equation~\eqref{eqG1}. 
Written for $xt$ instead of $x$, it reads
$$
t (Q(xt,0)-Q(0,0))+t(1+y)Q(0,y)=xty,
$$
with $y=Y_0(xt)$. Thus the equation we have obtained for $Q(xt,0)$
gives a cubic equation for $Q(0,y)$, and eliminating $x$ between this
equation and $ K(xt,y)=0$ gives a (still cubic) equation for
$Q(0,y)$ over $\qs(Z,y)$,  valid for a generic value of $y$.

\medskip
The equations for $Q(xt,0)$ and $Q(0,y)$ take a few Maple lines each
(four or five), and we do not give them explicitly.  They can be
seen in the accompanying  \href{http://www.labri.fr/perso/bousquet/publis.html}{Maple session}.
They factor if we express $x$ and $y$ in terms of van Hoeij's series $U$ and $V$,
as described in  Theorem~\ref{thm:main}.
Our alternative parametrization, given below the theorem, was obtained
thanks to 
the {\tt parametrization} 
command in Maple. One can use this command with $Z$ as a parameter, but it is faster to compute a
parametrization for a few values of $Z$ (e.g. $Z=10$, $Z=100$) and then reconstruct a generic
one.  
\qed

\section{More algebraic models}

It is natural to ask when the symmetric functions of the roots of the
  kernel are polynomials in $\bx$, since this property plays  a crucial role in our
proof. Let us consider each of the 23 models with small steps and a finite
group~\cite{BoMi10}, and  denote by $Y_0$
and $Y_1$ the  roots of the kernel
$$
K(x,y)=1-t \sum_{(i,j) \in \cS}x^i y^j.
$$
It is not hard to see that 
$Y_0+Y_1$ and $Y_0Y_1$ are polynomials in $\bx$ if and only if the
only step of the form $(i,1)$ is $(1,1)$. We  find 4 models having this
property.
\begin{itemize}
\item The first one is  Gessel's model $\cS=\{
 \rightarrow,\nearrow,\leftarrow, \swarrow\}$, which we have just solved.
\item
Then comes Kreweras' model $\cS=\{\leftarrow, \downarrow, \nearrow\}$,
which is also algebraic, and was
solved in~\cite{bousquet-versailles} using the same principles as
in this paper (see~\cite{Bous05} for a variant). The solution is simpler
than in Gessel's case in two aspects: first, the diagonal symmetry
implies that $Q(x,0)=Q(0,x)$, so that we have only one unknown series
$R(x)$, not $R(x)$ and $S(y)$ as before; then, and more importantly, the equation obtained by
forming a symmetric function of $Y_0$ and $Y_1$ only involves
$R(x)$, not $R(\bx)$.
\item Finally, we also have the models $\{\leftarrow,  \searrow, \nearrow\}$ and $\{\leftarrow,
  \rightarrow,  \searrow, \nearrow\}$, which are known to be
  D-finite~\cite{BoMi10} but transcendental (this can be derived from~\cite[Thm.~4]{BoRaSa12}).
\end{itemize}
More recently, Kauers and Yatchak initiated a study of walks in the
quadrant with multiple steps~\cite{kauers-yatchak}. Such walks
naturally arise, after projection,
in the study of 3D walks confined to the first octant~\cite{BoBoKaMe}. In
particular, it was proved in~\cite{BoBoKaMe}, using computer algebra,
that the model  $\{\leftarrow, \swarrow, \searrow, \rightarrow,
\rightarrow, \nearrow\}$ is algebraic (note the double East step). This was generalized by Kauers
and Yatchak~\cite{kauers-yatchak}, who proved that algebraicity persists if one includes a
South step, with arbitrary multiplicity~$\lambda$. For this model,
the symmetric functions of the roots of the kernel are polynomials in
$\bx$, and we can solve it using the  tools of this
paper. In fact, the proof is only marginally more difficult than in
Kreweras' case: there are two unknown series $R(x)$ and $S(y)$, but the equations obtained by
forming symmetric functions of $Y_0$ and $Y_1$ only involve
$R(x)$, not $R(\bx)$.

 Let us briefly sketch the main steps of the
solution. The basic equation reads:
\beq\label{eq-mult}
K(x,y) xy Q(x,y)= xy -R(x)-S(y),
\eeq
where $K(x,y)=1-t(\bx+\bx\by +\lambda\by +x\by +2x +xy)$ is the
kernel, 
\beq\label{RSdef}
R(x)=t(1+\lambda x +x^2) Q(x,0)-tQ(0,0) \quad \hbox{and} \quad
S(y)=t(1+y)Q(0,y).
\eeq
The group of this walk has order 6, and the orbit of $(x,Y_0)$
contains exactly 4 pairs $(x',y')$ for which $Q(x',y')$ is well defined:
$$
(x,Y_0), \quad \left( \frac{t(1+Y_1)}{1+\lambda t}, Y_0\right), \quad
 \left( \frac{t(1+Y_0)}{1+\lambda t}, Y_1\right), \quad 
\left( \frac{t(1+Y_0)}{1+\lambda t}, x/t+\lambda x -1\right).
$$
Each such pair gives an equation $x'y' =R(x')+S(y')$. The first
equation expresses $S(Y_0)$ in terms of $R(x)$, and a combination of
the last two expresses $S(Y_1)$ in terms of $S(x/t+\lambda x -1)$. By
extracting the positive part in the resulting expression of  $S(Y_0)+S(Y_1)$, we
conclude that
\beq\label{pm}
S(Y_0)+S(Y_1)= \bx \qquad\hbox{and} \qquad x = S(x/t+\lambda x -1)-R(x).
\eeq
From this we form a second symmetric function of $Y_0$ and $Y_1$:
$$
(S(Y_0)-xY_0)(S(Y_1)-xY_1)=-R(x)\left( R(x)+2x+2\bx -1/t\right).
$$
Extracting the positive part in $x$ gives
$$
\lambda x+x^2= -R(x)\left( R(x)+2x+2\bx -1/t\right)+ 2R'(0).
$$
This discrete differential equation can be solved using the method of
Section~\ref{sec:quad}. One obtains a cubic equation for $R'(0)$, and one of
degree 6 for $R(x)$. Finally, saying that $R(x)$ equals $-tQ(0,0)$ modulo
$(1+\lambda x+x^2)$ (see~\eqref{RSdef}) gives, by taking the above
equation modulo $(1+\lambda x+x^2)$, 
$$
t^2Q(0, 0)^2+(2\lambda t+1)Q(0, 0)=2R'(0)+1,
$$
and proves the algebraicity of $Q(0,0)$. One can then compute an
equation for $S(y)$ using the second part of \eqref{pm}. The
algebraicity of $Q(x,y)$ follows using the equation~\eqref{eq-mult} we
started from. \qed

\bigskip
To complete our guided tour of (conjecturally) algebraic models, let us mention   three  models that were discovered by Kauers
and Yatchak~\cite{kauers-yatchak}:
\begin{center}
  \begin{tikzpicture}[scale=.3] 
    \draw[->] (0,0) -- (-1,0) node[left] {$\scriptstyle 1$};
    \draw[->] (0,0) -- (-1,1) node[left] {$\scriptstyle 1$};
    \draw[->] (0,0) -- (0,1) node[above] {$\scriptstyle 2$};
    \draw[->] (0,0) -- (1,1) node[right] {$\scriptstyle 1$};
    \draw[->] (0,0) -- (1,0) node[right] {$\scriptstyle 2$};
    \draw[->] (0,0) -- (1,-1) node[right] {$\scriptstyle 1$};
    \draw[->] (0,0) -- (0,-1) node[below] {$\scriptstyle 1$};
  \end{tikzpicture}\hfil
  \begin{tikzpicture}[scale=.3] 
    \draw[->] (0,0) -- (-1,-1) node[left] {$\scriptstyle 1$};
    \draw[->] (0,0) -- (-1,0) node[left] {$\scriptstyle 2$};
    \draw[->] (0,0) -- (-1,1) node[left] {$\scriptstyle 1$};
    \draw[->] (0,0) -- (0,1) node[above] {$\scriptstyle 1$};
    \draw[->] (0,0) -- (1,0) node[right] {$\scriptstyle 1$};
    \draw[->] (0,0) -- (1,-1) node[right] {$\scriptstyle 1$};
    \draw[->] (0,0) -- (0,-1) node[below] {$\scriptstyle 2$};
  \end{tikzpicture}\hfil
  \begin{tikzpicture}[scale=.3] 
    \draw[->] (0,0) -- (-1,-1) node[left] {$\scriptstyle 1$};
    \draw[->] (0,0) -- (-1,0) node[left] {$\scriptstyle 2$};
    \draw[->] (0,0) -- (-1,1) node[left] {$\scriptstyle 1$};
    \draw[->] (0,0) -- (0,1) node[above] {$\scriptstyle 2$};
    \draw[->] (0,0) -- (1,0) node[right] {$\scriptstyle 1$};
    \draw[->] (0,0) -- (1,1) node[right] {$\scriptstyle 1$};
    \draw[->] (0,0) -- (0,-1) node[below] {$\scriptstyle 1$};
  \end{tikzpicture}
\end{center}
The weights indicate multiplicities. All three models have a group of
order 10. We can solve the first two  using half-orbit sums, as
in~\cite[Sec.~6]{BoMi10}, and we find them to be algebraic indeed.
Regarding the third model,   its algebraicity should follow from a
complex analysis approach~\cite{kilian-private}, as for Gessel's model in~\cite{BoKuRa13}. 

Finally, let us recall that another natural way of weighting steps
comes from the study of Markov chains in the
quadrant~\cite{FaIaMa99}. In this setting each step is weighted by its
probability. It may be possible to construct a Markov chain with Gessel's steps
and stationary distribution given by an algebraic \gf, provable by the
tools of this paper. This would parallel the case of Kreweras' walks:
they first appeared as a counting problem  in
1965~\cite{kreweras}, and then 20 years later, independently, 
 as a  Markov chain with algebraic stationary
 distribution~\cite{flatto-hahn}. Many years later, the algebraicity
 of both problems was proved by a common elementary approach~\cite{Bous05}, using
 ideas that we have extended in this paper.

\bigskip
\noindent{\bf Acknowledgements.} The author is grateful to Alin
Bostan, Tony Guttmann
and Kilian Raschel for various comments on this manuscript. She also
thanks the organizers of the workshop ``Approximation and
Combinatorics'' in 2015 in CIRM (Luminy, F) Herwig Hauser and Guillaume Rond,  for very
interesting discussions on the algebraic aspects of quarter plane equations.

 \bibliographystyle{plain}
\bibliography{qdp}


\end{document}